\documentclass[18pt]{amsart}
\usepackage{amssymb}
\usepackage{amscd}
\usepackage{lscape}
\usepackage[arrow,matrix,curve]{xy}
\usepackage{times}
\usepackage{mathrsfs}
\usepackage{upgreek}
\usepackage{stmaryrd}
\usepackage[colorlinks = true]{hyperref}
\usepackage{rotating}
\usepackage{tikz}
\usetikzlibrary{cd}
\usepackage{pst-node}
\usepackage{tikz-cd} 
\usepackage[arrow,matrix,curve]{xy}
\usepackage{times}
\usepackage{graphicx}

\date{\today}

\DeclareMathOperator{\Hom}{Hom}

\makeatletter
\def\lotimes{\@ifnextchar_{\@lotimessub}{\@lotimesnosub}}
\def\@lotimessub_#1{\mathchoice{\mathbin{\mathop{\otimes}^L}_{#1}}%
  {\otimes^L_{#1}}{\otimes^L_{#1}}{\otimes^L_{#1}}}
\def\@lotimesnosub{\mathbin{\mathop{\otimes}^L}}
\makeatother

\numberwithin{equation}{section}
\newtheorem{thm}{Theorem}[section]
\newtheorem{lem}[thm]{Lemma}

\newtheorem{cor}[thm]{Corollary}

\theoremstyle{definition}
\newtheorem{defn}[thm]{Definition}

\theoremstyle{remark}

\numberwithin{figure}{section}
\numberwithin{table}{section}

\title{A Hopf algebra from preprojective modules}

\author{Pak-Hin Li}
\email{pl473@cornell.edu}

\begin{document}

\begin{abstract}
Let $Q$ be a finite type quiver i.e. ADE Dynkin quiver.
Denote by $\Lambda$ its preprojective algebra. It is known that there are finitely many indecomposable $\Lambda$-modules if and only if $Q$ is of type $A_1,A_2,A_3,A_4$.
In this paper, extending Lusztig's construction of $U\frak{n}_+$, we study an algebra generated by these indecomposable submodules. 
It turns out that it forms the universal enveloping algebra of some nilpotent Lie algebra inside the function algebra on Lusztig's nilpotent scheme. 
The defining relations of the corresponding nilpotent Lie algebra for type $A_1, A_2,A_3,A_4$ are given here.

\end{abstract}

\maketitle

\section{Introduction}
\label{s:intro}
In this paper, we will follow the notation in \cite{GLS05} closely and the quiver we consider will always be of type $A_n$ where $n\in\{1,2,3,4\}$. 
\subsection{Lusztig's nilpotent scheme}
Given a simply laced $A,D,E$ type quiver, we denote the vertex set by $I$. 
Let $\bar{Q}$ be the quiver obtained from the Dynkin diagram of $\mathfrak{g}$ by replacing each edge by a pair of opposite arrows $(a,a^*)$. 
\textbf{The preprojective algebra} $\Lambda$ is defined as the quotient of the path algebra $\mathbb{C}\bar{Q}$ by the two-sided ideal generated by $\sum(aa^*-a^*a)$ where the sum is over all pairs of arrows with opposite direction.
Given a finite dimensional vector space $V$ with graded dimension $|V|$ (associate an integer at each vertex of the Dynkin quiver), we can construct \textbf{ Lusztig's nilpotent scheme} $\Lambda_V$ which is defined as a subscheme of $\prod_{e\in \text{edges}}\Hom(V_{e_{\text{tail}}},V_{e_{\text{head}}}) $ by these preprojective relations. 
For type $A_n$, these relations are as follows.
For a $I$-graded vector space $V$ with dimension vector $|V|=(d_1,d_2,...,d_n)$, we can define $\Lambda_V=\{((f_i)_{i=1,2.,,n-1},(q_i)_{i=1,2,...,n-1})\in \prod_{i=1}^{n-1}\Hom(\mathbb{C}^{d_i},\mathbb{C}^{d_{i+1}}) \times \prod_{i=1}^{n-1}\Hom(\mathbb{C}^{d_{i+1}},\mathbb{C}^{d_{i}}) : f_iq_i=q_{i-1}f_{i-1},i=1,2,...,n   \}$, where the maps are zero if the index is not in $\{1,2...,n-1\}$. 
There is a group $G_V:=\prod_{i\in I} GL(V_i)$ acting on $\Lambda_V$ as follows. 
Let $x=(x_\alpha)_{\alpha\in \text{edge} }\in \Lambda_V$ and $g=(g_i)_{i\in I}\in G_V$. 
Then $g\cdot x= (y_\alpha)_{\alpha \in \text{edge}}$ is given by:
$y_\alpha= g_{\text{tail}(\alpha)}x_{\alpha} g_{\text{head}(\alpha)}^{-1}$.

\subsection{The function algebra on Lusztig's nilpotent scheme}
Define $M(\Lambda_V)$ to be the space of all constructible functions on $\Lambda_V$ and denote by $M(\Lambda_V)^{G_V}$ the $G_V$-invariant functions.
Let $\widetilde{\mathcal{M}}=\oplus_{|V|\in {{\mathbb{N}^I}}} M(\Lambda_V)^{G_V}$, where $\mathbb{N}$ denotes the set of nonnegative integers.
One can define an algebra structure $(\widetilde{\mathcal{M}},*)$ on $\widetilde{\mathcal{M}}$ as follows.
If $V_1,V_2,V_3$ are $I$-vector spaces and $|V_1|+|V_2|=|V_3|$, we define $*:M(\Lambda_{V_1})^{G_{V_1}}\times M(\Lambda_{V_2})^{G_{V_2}}\rightarrow M(\Lambda_{V_3})^{G_{V_3}}$ by $1_{O_1}* 1_{O_2}(x):=\chi(\Phi_{O_1,O_2,x})$ where $$\Phi_{O_1,O_2,x}=\{U\in Gr(|V_2|,V_3): U \text{ is a $\Lambda$-submodule of $x$ of type $O_2$ }, x/U \text{ has type $O_1$}  \}$$
and $\chi$ is compactly supported Euler characteristic.
When the quiver is of type $A_1,A_2,A_3,A_4$ and only then, the number of indecomposable modules is finite and as a result $M(\Lambda_v)^{G_v}$ is finite dimensional.
From here, the \underline{compactly supported} Euler characteristic of a set $A$ is denoted by $\chi(A)$.

\subsection{The coproduct structure on the function algebra $\widetilde{\mathcal{M}}$}
\label{coproduct}
In \cite{GLS05}, given $V',V'',V$ such that $|V|=|V'|+|V''|$, there is defined a map $\text{Res}_{V',V''}^{V}: M(\Lambda_V)^{G_V}\rightarrow M(\Lambda_{V'}\times \Lambda_{V''})^{G_{V'}\times G_{V''}}$ as follows: $$\text{Res}_{V',V''}^{V}(f)(x',x'')=f(x'\oplus x''),$$ 
where $(x',x'') \in \Lambda_{V'}\times \Lambda_{V''}$.
It is proven in \cite{GLS05} that the map $\text{Res}_{V',V''}^{V}$  descends to  $M(\Lambda_{V'})^{G_{V'}} \otimes M( \Lambda_{V''})^{G_{V''}}$ when restricted to the algebra generated by $\{e_i=1_{Z[i]}: i \in I\}$ (where $Z[i]$ is the $\Lambda$-module the dimension vector of which is Kronecker delta $(\delta_{ij})_{j\in I}$) and it agrees with the coproduct structure of $U\mathfrak{n}_{+}$.  
In the next section, it will be shown that $\text{Res}_{V',V''}^{V}$ can actually give a coproduct structure to the larger subalgebra $\mathcal{A}$ generated by indecomposable modules.

\subsection{Acknowledgements}
I would like to thank my advisor Allen Knutson for suggestions and ideas of this project. The calculation of Lie algebra cohomology, lower and upper central series were done using Maple.

\section{Main result}
\label{sec:statement}
Hereafter, $Q$ is one of $A_1,A_2,A_3,A_4$.
Define $\mathcal{A}$ to be the subalgebra of $\widetilde{\mathcal{M}}$ which is generated by the characteristic functions $\{1_K: \text{ $K$ is an orbit of an indecomposable module}\}$. 
We will see later in this section that actually $\mathcal{A}$ is equal to $\widetilde{\mathcal{M}}$.

\subsection{Coproduct structure on $\mathcal{A}$}
\label{ss:notation}
In this subsection, we prove that $Res$ gives us a coproduct structure on $\mathcal{A}$ and the coproduct is multiplicative, so $\mathcal{A}$ is necessarily the universal enveloping algebra of its primitive elements.  

\begin{lem}
\label{basic_euler}
Let $f=1_{K_1}*1_{K_2}*...*1_{K_n} $ where $K_i$ are $G_{V_i}$-invariant subsets of $\Lambda_{V_i}$. Then $f(x)=\chi(\{(V^0\supseteq V^1\supseteq ...\supseteq V^n=0):x(V^i)\subset V^i, [V^{i-1}/V^i]\in K_i, \text{for $i=1,2,...,n$ }\})$.
\end{lem}
\begin{proof}
This follows directly from the definition of product structure on $\widetilde{\mathcal{M}}$.
\end{proof}

\begin{defn}
\label{important}
Let $U_1,...,U_r$ be the different isomorphism classes of indecomposable $\Lambda$-module. 
Each function in $M(\Lambda_V)^{G_V}$ can be expressed as a linear combination of indicator functions $1_{a_1U_1\oplus...\oplus a_r U_r}$. 
Let $Y\subset \Lambda_V$ be the $G_V$-orbit of $a_1U_1\oplus...\oplus a_r U_r$. 
We define the indicator functions $1_Y$ and $1_{U_i}$ to be $I(a_1,...,a_r)$ and $I_i$ respectively.
We also let $a=(a_1,...,a_r)$ and denote $I(a_1,...,a_r)$ by $I(a)$.
\end{defn}

\begin{lem}
For $Q=A_1,A_2,A_3,A_4$, the map $\text{Res}_{V',V''}^{V}: M(\Lambda_V)^{G_V}\rightarrow M(\Lambda_{V'}\times \Lambda_{V''})^{G_{V'}\times G_{V''}}$ factors through $M(\Lambda_{V'})^{G_{V'}} \otimes M(\Lambda_{V''})^{G_{V''}}$. 
Hence the map $Res$ induces a map $\Delta:\widetilde{\mathcal{M}}\rightarrow \widetilde{\mathcal{M}}\otimes \widetilde{\mathcal{M}}$.
\end{lem}
\begin{proof}
Each function in $M(\Lambda_V)^{G_V}$ is a linear combination of $1_{K}$ where the $K$ are $G_V$-orbits of $\Lambda$-modules. 
Say that module is of the form $a_1V_1\oplus a_2 V_2...\oplus a_rV_r$ where $V_i$ are different isomorphism types of indecomposable module. 
Then by the Krull-Schmidt theorem, every $\Lambda$-module has a unique decomposition into indecomposable modules, so we get:
\begin{align*}
Res_{V',V''}^{V}(1_K)(x'\oplus x'')=\sum_{{}^{\hspace{0.8cm}(b_1,b_2,...,b_r):}_{b_1v_1+...+b_rv_r=|V'|,b_i\le a_i}} 1_{b_1V_1\oplus...\oplus b_rV_r}(x') 1_{(a_1-b_1)V_1\oplus...\oplus (a_r-b_r)V_r}(x'')
\end{align*}
Hence, $Res$ induces a map $\Delta: \widetilde{\mathcal{M}}\rightarrow \widetilde{\mathcal{M}}\otimes \widetilde{\mathcal{M}}$.
\end{proof}

\begin{lem}
\label{fixedpoint}
Let $X\subset Y$ be quasiprojective and $T$ be a torus acting on $Y$ algebraically and preserving $X$ where $Y$ is a partial flag variety. 
Then the compactly supported Euler characteristic of $X$ is equal to the compactly supported Euler characteristic of the $T$-fixed points in $X$.
\end{lem}
\begin{proof}
This standard fact follows from the Bia\l ynicki-Birula decomposition \cite{BB73} and the additivity of compactly supported Euler characteristic on locally closed sets in an algebraic variety.
\end{proof}

We need a basic Lemma here.

\begin{lem}
\label{basic}
Let $V=\bigoplus_{i=1}^{N} V_i$ be the isotopic decomposition of a $T$-representation where $T=(\mathbb{C}^*)^{N}$. 
If $U$ is a $T$-invariant subspace of $V$, then $U=\oplus_{i=1}^{N} (U\cap V_i)$.
\end{lem}

\begin{lem}
\label{submodule}
Let $V=V_1\oplus V_2...\oplus V_t$, where $V_i$ are indecomposable $\Lambda$-modules (possibly with repetition).
Let $T=(\mathbb{C}^*)^{t}$ be the torus acting on $x$ by scaling each indecomposable summand. 
If $U$ is an indecomposable submodule in $x$ such that $U$ is $T$-invariant, then $U$ is a submodule of one of the $V_i$.
\end{lem}
\begin{proof}
Since $U$ is $T$-invariant, then by Lemma \ref{basic}, we must have $U=\oplus_{i=1}^r (U\cap V_i)$. 
$U$ is a $\Lambda$-module so $U\cap V_i$ is also a $\Lambda$-module. 
As a result, $U$ is a direct sum of $\Lambda$-modules. 
However, $U$ is indecomposable so $U$ is a submodule of one of the $V_i$.
\end{proof}


\begin{defn}
Given a sequence indecomposable modules $U_{j_1},U_{j_2},....,U_{j_t}$, define $d_{J}:=I_{j_1}*....*I_{{j_t}}$  where $J=(j_1,j_2,...,j_t)$. 
Given $c\in \{0,1\}^t$, we also define $d_{J,c}$ to be the product of $I_{j_i}$ where $c_i=1$. 

\end{defn}

We are now going to show that $\Delta: \mathcal{A} \to \mathcal{A} \otimes \mathcal{A}$ is a algebra homomorphism.
\begin{lem}
\label{key}Let $J=(j_1,...,j_N)$ be a sequence of indices of indecomposable modules.
Following section \ref{coproduct}, let $V=V'\oplus V''$ and $x',x''$ be $\Lambda$-modules with dimension vectors $|V'|,|V''|$ respectively.
We have $Res_{V',V''}^{V} d_{J}(x',x'')=\sum_{(c',c'')} d_{J,c'}(x')d_{J,c''}(x'') $ , where $(c',c'')$ runs through all pairs of sequences in $\{0,1\}^N\times\{0,1\}^N$ which satisfy $c_i'+c_i''=1$ for all $i=1,2,...,N$ , $\sum_{i=1}^{N}c_i'|U_{j_i}|=|V'|$ and $\sum_{i=1}^{N}c_i''|U_{j_i}|=|V''|$. 
\end{lem}

\begin{proof}
By definition of $Res_{V',V''}^{V}$ and Lemma \ref{basic_euler} , we have $Res_{V',V''}^{V} d_{J}(x',x'')=d_J(x'\oplus x'')=\chi(\{V=V^0\supseteq V^1 \supseteq ....\supseteq V^N=0: x(V^i) \subset V^i, [V^{i-1}/V^{i}] \in K_{j_i}\})$.
Let $T=\mathbb{C}^{*}\times \mathbb{C}^{*}$ be a torus acting on $V=V'\oplus V''$ by scaling summand $V'$ and $V''$ and it induces an action on $\{V=V^0\supseteq V^1 \supseteq ....\supseteq V^N=0: x(V^i) \subset V^i, [V^{i-1}/V^{i}] \in K_{j_i}\}$.
Let $X=\{V=V^0\supseteq V^1 \supseteq ....\supseteq V^N=0: x(V^i) \subset V^i, [V^{i-1}/V^{i}] \in K_{j_i}\}$.
Then by Lemma \ref{fixedpoint}, we have $d_J(x'\oplus x'')=\chi(X)=\chi(X^T)$.
By Lemma \ref{submodule}, for each flag $(V^i)_{i=0,1,...,N}$ in $X^T$, we have  $V^i= V^i \cap V'\oplus V^i\cap V''$.
Then $V^{i-1}/V^{i}\cong V^{i-1}\cap V''/V^{i}\cap V'' \oplus (V^{i-1}/V'')/(V^{i}/V'')$.
As a result,  $V^{i-1}\cap V''/V^{i}\cap V'' $ is either trivial or isomorphic to $V^{i-1}/V^{i}$ since $V^{i-1}/V^{i}$ is indecomposable.
Then following  \cite{GLS05} Lemma 6.1, we have $Res_{V',V''}^{V}d_{J}(x',x'')=\sum_{(c',c'')}d_{J,c'}(x')d_{J,c''}(x'')$ where $(c',c'')$ runs through all pairs of sequences in $\{0,1\}^N \times \{0,1\}^N$ which satisfy $c_i'+c_i''=1$ , $\sum_{i=1}^{N}c_i'|U_{j_i}|=|V'|$ and $\sum_{i=1}^{N}c_i''|U_{j_i}|=|V''|$.

\end{proof}

\begin{cor}
\label{coprod}
The map $\Delta: \mathcal{A} \rightarrow \mathcal{A}\otimes \mathcal{A}$ is multiplicative i.e. $\Delta(fg)=\Delta(f)\Delta(g)$ , coassociative and cocommutative.   
As a result, $(\mathcal{A},*,\Delta)$ is a Hopf algebra.
\end{cor}
\begin{proof}
It directly follows from Lemma \ref{key}.

\end{proof}

\begin{lem}
\label{form} 
Let $r$ be the number of indecomposable modules for $Q$ ($1,4,12,40$ for $A_{1,2,3,4}$).
Denote by $U_i$ the i-th indecomposable $\Lambda$-module and order the indecomposable modules such that if  
$U_i$ is a submodule of $U_j$, then $i\le j$ by a refinement of total dimension. 
Follwoing definition \ref{important}, we have $$I(a_1,a_2,...,a_r) * I_j = (a_j+1) I(a+e_j)+\sum_{b\in S } c_b I(b),$$ where $e_j=(\delta_{ij})_{i=1}^{r}$ and $S=\{(b_1,...,b_r) \in \mathbb{N}^{r}: b_{i_0}=a_{i_0}+1 \text{ for some $i_0>j$}, \sum_{i}b_i|U_i|=|V''|\}$.
\end{lem}

\begin{proof}
Let $V,V'$ be the ambient vector spaces such that $I(a_1,a_2,...,a_r) \in \Lambda_{V}$, $I_j \in \Lambda_{V'}$. 
Let $V''=V\oplus V'$ and  $x=b_1U_1\oplus b_2U_2\oplus..\oplus b_r U_r \in \Lambda_{V''}$.
Suppose $1_{b_1U_1\oplus b_2U_2\oplus..\oplus b_r U_r}$ is a nonzero summand with coefficient $c_b$.
Then $I(a_1,a_2,...,a_r)*I_j(x)=c_b$.

By the definition of product structure on $\tilde{\mathcal{M}}$,
 \begin{align*}c_b=\chi(\{U\in Gr(|V'|,V''): U\cong U_j, x/U \cong a_1U_1\oplus a_2U_2\oplus...\oplus a_r U_r \}).\end{align*}
Let $T=(\mathbb{C}^{*})^{\sum_{i}b_i}$ act on $x$ by scaling each indecomposable module independently.
The action is well-defined since it preserves each preprojective relation.
Then by Lemma \ref{fixedpoint}, $c_b=\chi(\{U\in Gr(|V'|,V''): U\cong U_j, x/U \cong \oplus_{i=1}^{r}a_iU_i \}^T)$.
Let $U\in \{U\in Gr(|V'|,V''): U\cong U_j, x/U \cong \oplus_{i=1}^{r}a_iU_i \}^T$.
Since $U$ is invariant under $T$ and is isomorphic to an indecomposable module $U_j$,  by Lemma \ref{submodule}, $U$ is a $\Lambda$-submodule of one of the $U_i$ summands of $x=\oplus_{i}b_i U_i$.
Since $c_b \neq 0$, there is an integer $i_0$ between $1$ and $r$ such that $U$ is a submodule of $U_{i_0}$ as a summand of $x$.
From the ordering of $U_i$, we have $i_0\ge j$.
Since $x/U\cong\oplus_{i=1}^{r}a_iU_i$, we have $b_1U_1\oplus b_2U_2\oplus...\oplus(b_{i_0}-1) U_{i_0}\oplus...\oplus b_r U_r\oplus(U_{i_0}/U)\cong \oplus_{i=1}^{r}a_iU_i$.
Then we have $b_{i_0}-1=a_{i_0}$ since $U_{i_0}/U$ does not have summand isomorphic to $U_{i_0}$.
If $i_0=j$, $c_b=a_j+1$.
As a result, we have:
$$I(a_1,a_2,...,a_r) * I_j = (a_j+1) I(a+e_j)+\sum_{b\in S } c_b I(b),$$
where $S=\{(b_1,...,b_r) \in \mathbb{N}^{r}: b_{i_0}=a_{i_0}+1 \text{ for some $i_0>j$}, \sum_{i}b_i|U_i|=|V''|\}$.

\end{proof}

\begin{lem}
$\mathcal{A}$ is equal to $\widetilde{\mathcal{M}}$.
\end{lem}
\begin{proof}
$\widetilde{\mathcal{M}}$ is graded by dimension so it is $ \mathbb{N}^{I}$-graded.
It suffices to show that for each dimension vector $v=(v_i)_{i\in I}\in \mathbb{N}^{I}$, $\mathcal{A}_{v}=\widetilde{\mathcal{M}}_v$. 
In the base case when the sum of $v$ is equal to $1$, $\mathcal{A}_{v}=\widetilde{\mathcal{M}}_v$ because both are $1$-dimensional. 
We proceed to show $\mathcal{A}_{v}=\widetilde{\mathcal{M}}_v$ by induction.
Suppose $\mathcal{A}_{v}=\widetilde{\mathcal{M}}_v$ for all $v$ such that $\sum_{i\in I} v_i<N$.
Now let $v=(v_i)_{i\in I}\in \mathbb{N}^{I}$  such that $\sum_{i\in I} v_i=N$. 
Using the notation in  Definition \ref{important}, It suffices to show that for any $I(a_1,...,a_r) \in \widetilde{\mathcal{M}}_v$, we have $I(a_1,...,a_r)\in \mathcal{A}_v$. 
We can further order the indecomposable modules such that if $U_i$ is a submodule of $U_j$, then $i\le j$ by a refinement of total dimension.
Now we fix $N$. 
Let $m$ be the largest index such that $a_m$ is not zero (i.e. $(a_1,a_2,...,a_r)=(a_1,...,a_m,0,..,0)$).
We can further use induction on $m$ to show $I(a_1,...,a_r)\in \mathcal{A}_v$.
For the base case, when $m=r$, by Lemma \ref{form}, $I(a_1,...,a_r-1)*I_r=a_r I(a_1,...,a_r)$.
By induction assumption $I(a_1,...,a_r-1),I_r\in \mathcal{A}$ so $I(a_1,...,a_r)\in \mathcal{A}_v$.
Now assume that $I(a_1,...,a_r)\in \mathcal{A}_v$ whenever the largest index $m$ such that $a_m> 0$ is greater than $M$.
When $m=M$, we want to show that $I(a_1,...,a_M,0,...,0)\in \mathcal{A}_v$. 
We can consider $I(a_1,...,a_M-1,0,...,0)*I_M$.
By Lemma \ref{form}, $I(a_1,...,a_M-1,0,...,0)*I_M=a_M I(a_1,a_2,...,a_M,0,...,0)+L$, where $L$ is a linear combination of other $I(a_1',...,a_r')$ whose highest index $m$ such that $a_m'>0$ is larger than $M$. 
By induction assumption, we have $I(a_1,...,a_m,0,...,0)\in \mathcal{A}_v$ if $m>M$ and $a_m>0$. 
As a result, $L\in \mathcal{A}_v$, and $I(a_1,a_2,...,a_M,0,...,0)=(I(a_1,...,a_M-1,0,...,0)*I_M-S)/a_M \in \mathcal{A}_v$ since by induction we also have $I(a_1,...,a_M-1,0,...,0)\in \mathcal{A}$.

\end{proof}

\begin{lem}[Milnor-Moore theorem \cite{MM65}]
\label{MM}
Given a connected graded cocommutative Hopf algebra $A$ over a field of characteristic zero with $ \dim A_n < \infty$, the natural Hopf algebra homomorphism $ U(P(A)) \to A$ from the universal enveloping algebra of the graded Lie algebra $P(A)$ of primitive elements of $A$ is an isomorphism. 
\end{lem}

\begin{thm}
For $Q=A_1,A_2,A_3,A_4$, let $\mathcal{A}$ be the subalgebra generated by the orbit of indecomposable preprojective modules. Then $\mathcal{A}$ is isomorphic to $U\widetilde{\mathfrak{n}}$ for some Lie algebra $\widetilde{\mathfrak{n}}$ with a basis indexed by indecomposable modules.
\end{thm}
\begin{proof}
From Lemma \ref{coprod}, the map $\Delta$ is multiplicative on $\mathcal{A}\otimes \mathcal{A}$. 
We can conclude that $(\mathcal{A},\cdot,\Delta)$ is a Hopf algebra. 
From the definition of $\Delta$, we can see that the primitive elements are of the form $1_K$, where $K$ is an orbit of an indecomposable $\Lambda$-module.
Then by Lemma \ref{MM}, $\mathcal{A}$ is isomorphic to the universal enveloping algebra $U\widetilde{\mathfrak{n}}$ where $\widetilde{\mathfrak{n}}$ is freely spanned by primitive elements which are indicator functions on orbits of indecomposable modules.
\end{proof}


\subsection{Indecomposable modules for $Q=A_1,A_2,A_3,A_4$}
\label{table}
It is well-known that there are $1,4,12,40$ indecomposable modules for $Q=A_1,A_2,A_3,A_4$ respectively, described in \cite{GLS05}. 
For the sake of completeness, I will also describe and label all of them here. 
The dimension vectors are represented by the number at each vectex and if there is no arrow between two numbers, it means the map between two vertices is a zero map.
\\
For $Q=A_2$:
\begin{align*}
U_1=[1\hspace{5mm} 0]\quad U_2=[0\hspace{5mm}1]\quad U_3&=[1\to 1]\quad U_4=[1\leftarrow 1].
\end{align*}
For $Q=A_3$:
\begin{align*}
U_1&=[1\hspace{5mm} 0\hspace{5mm}  0 ] \quad U_2=[0\hspace{5mm} 1\hspace{5mm}  0 ]\quad U_3=[0\hspace{5mm} 0\hspace{5mm}  1 ]\quad
U_4=[1\to 1\hspace{5mm} 0]\quad U_5=[0\hspace{5mm} 1\to 1]\\
 U_6&=[1\leftarrow 1 \hspace{5mm} 0] \quad
 U_7=[0 \hspace{5mm} 1\leftarrow 1]\quad U_8=[1\to 1 \to 1]\quad U_9=[1\leftarrow 1 \leftarrow 1]\quad U_{10}=[1\rightarrow 1\leftarrow 1]\\
  U_{11}&=[1\leftarrow 1\rightarrow 1] \quad
U_{12}=
\left[
\begin{matrix}
 & & 1 & &\\
 &\swarrow && \searrow \\
 1 &&&& 1 \\
 &\searrow && \swarrow \\
 & & 1 & &\\
\end{matrix}
\right].
\end{align*}
For $Q=A_4$:\\
\begin{align*}
U_1&=[1\hspace{5mm}0\hspace{5mm}0\hspace{5mm}0]\quad U_2=[0\hspace{5mm}1\hspace{5mm}0\hspace{5mm}0]\quad U_3=[0\hspace{5mm}0\hspace{5mm}1\hspace{5mm}0]\quad U_4=[0\hspace{5mm}0\hspace{5mm}0\hspace{5mm}1]\quad U_5=[1\to 1 \hspace{5mm} 0 \hspace{5mm} 0] \\
U_6&=[1\leftarrow 1 \hspace{5mm} 0 \hspace{5mm} 0]\quad U_7=[0 \hspace{5mm}1\to 1 \hspace{5mm} 0]\quad U_8=[0 \hspace{5mm}1\leftarrow 1 \hspace{5mm} 0]\quad U_9=[0\hspace{5mm}0 \hspace{5mm}1\to 1]\quad U_{10}=[0\hspace{5mm}0 \hspace{5mm}1\leftarrow 1]\\
U_{11}&=[1\to 1 \to 1 \hspace{5mm} 0]\quad U_{12}=[1\leftarrow 1 \leftarrow 1 \hspace{5mm} 0]\quad U_{13}=[1\to 1 \leftarrow 1 \hspace{5mm} 0]\quad U_{14}=[1\leftarrow 1 \to 1 \hspace{5mm} 0]\\
U_{15}&=[0\hspace{5mm}1\to 1 \to 1]\quad U_{16}=[0\hspace{5mm}1\leftarrow 1 \leftarrow 1 ]\quad U_{17}=[0\hspace{5mm}1\to 1 \leftarrow 1]\quad U_{18}=[0\hspace{5mm}1\leftarrow 1 \to 1]\\
U_{19}&=
\left[
\begin{matrix}
 & & 1 & &&&\\
 &\swarrow && \searrow \\
 1 &&&& 1 &&0\\
 &\searrow && \swarrow \\
 & & 1 & &&&\\
\end{matrix}
\right]
\quad
U_{20}=\left[
\begin{matrix}
&& & & 1 & &\\
 &&&\swarrow && \searrow \\
 0&&1 &&&& 1 \\
 &&&\searrow && \swarrow \\
 &&& & 1 & &\\
\end{matrix}
\right] \quad 
U_{21}=[1\to 1\to 1 \to 1], U_{22}=[1\to 1\to 1 \leftarrow 1] \quad\\
U_{23}&=[1\to 1 \leftarrow 1 \to 1]\quad U_{24}=[1\to 1 \leftarrow 1 \leftarrow 1]\quad U_{25}=[1\leftarrow 1 \to 1 \to 1]\quad U_{26}=[1\leftarrow 1 \to 1 \leftarrow 1]\\
 U_{27}&=[1\leftarrow 1\leftarrow 1 \to 1]\quad U_{28}=[1\leftarrow 1 \leftarrow 1 \leftarrow 1]\quad 
 U_{29}=\left[
\begin{matrix}
 & & 1 & &&&\\
 &\swarrow && \searrow \\
 1 &&&& 1 &&\\
 &\searrow && \swarrow &&\searrow\\
 & & 1 & &&&1\\
\end{matrix}
\right]\quad
 U_{30}=\left[
\begin{matrix}
 & & 1 & &&&1\\
 &\swarrow && \searrow &&\swarrow\\
 1 &&&& 1 &&\\
 &\searrow && \swarrow \\
 & & 1 & &&&\\
\end{matrix}
\right] \quad\\
 U_{31}&=\left[
\begin{matrix}
&& & & 1 & &\\
 &&&\swarrow && \searrow \\
 &&1 &&&& 1 \\
 &\swarrow&&\searrow && \swarrow \\
 1&&& & 1 & &\\
\end{matrix}
\right]\quad
 U_{32}=\left[
\begin{matrix}
1&& & & 1 & &\\
 &\searrow&&\swarrow && \searrow \\
 &&1 &&&& 1 \\
 &&&\searrow && \swarrow \\
 &&& & 1 & &\\
\end{matrix}
\right]\quad
 U_{33}=\left[
\begin{matrix}
 & & 1 & &&&\\
 &\swarrow && \searrow \\
 1 &&&& 2 &&\\
 &\searrow && \swarrow &&\searrow\\
 & & 1 & &&&1\\
\end{matrix}
\right]\quad\\
 U_{34}&=\left[
\begin{matrix}
 & & 1 & &&&1\\
 &\swarrow && \searrow &&\swarrow\\
 1 &&&& 2 &&\\
 &\searrow && \swarrow \\
 & & 1 & &&&\\
\end{matrix}
\right]\quad
 U_{35}=\left[
\begin{matrix}
&& & & 1 & &\\
 &&&\swarrow && \searrow \\
 &&2 &&&& 1 \\
 &\swarrow&&\searrow && \swarrow \\
 1&&& & 1 & &\\
\end{matrix}
\right]\quad
 U_{36}=\left[
\begin{matrix}
1&& & & 1 & &\\
 &\searrow&&\swarrow && \searrow \\
 &&2 &&&& 1 \\
 &&&\searrow && \swarrow \\
 &&& & 1 & &\\
\end{matrix}
\right]\quad\\
 U_{37}&=\left[
\begin{matrix}
 & & 1 & &&&\\
 &\swarrow && \searrow \\
 1 &&&& 1 &&\\
 &\searrow && \swarrow &&\searrow\\
 & & 1 & &&&1\\
 & && \searrow &&\swarrow\\
   &&&& 1 &&\\
\end{matrix}
\right]\quad
 U_{38}=\left[
\begin{matrix}
&& & & 1 & &\\
 &&&\swarrow && \searrow \\
 &&1 &&&& 1 \\
 &\swarrow&&\searrow && \swarrow \\
 1&&& & 1 & &\\
 &\searrow&&\swarrow&&  \\
 &&1 &&&& 
\end{matrix}
\right]\quad
 U_{39}=\left[
\begin{matrix}
1&& & & 1 & &\\
 &\searrow&&\swarrow && \searrow \\
 &&2 &&&& 1 \\
 &\swarrow &&\searrow && \swarrow \\
 1&&& & 1 & &\\
\end{matrix}
\right]\quad\\
 U_{40}&=\left[
\begin{matrix}
 & & 1 & &&&1\\
 &\swarrow && \searrow &&\swarrow\\
 1 &&&& 2 &&\\
 &\searrow && \swarrow&&\searrow \\
 & & 1 & &&&1\\
\end{matrix}
\right].
\end{align*}

The most nontrivial indecomposable modules in $Q=A_4$ are $U_{39}$ and $U_{40}$ so I give their explicit maps here.\\

$
\begin{array}{ccccccccc}
&\begin{pmatrix}1&0\\-1&0\end{pmatrix}
&&\begin{pmatrix}0&0\\1&0\end{pmatrix}
&&\begin{pmatrix}1&0\end{pmatrix}\\
U_{39}=k^2&\rightleftarrows&k^2&\rightleftarrows&k^2&
\rightleftarrows&k\\
&\begin{pmatrix}0&0\\1&1\end{pmatrix}
&&\begin{pmatrix}1&0\\0&0\end{pmatrix}
&&\begin{pmatrix}0\\1\end{pmatrix}\\
\end{array}
$ \qquad
$U_{40}=U_{39} \text{  reversed}$

\newpage
\subsection{The Lie algebra $\widetilde{\mathfrak{n}}$ with basis indexed by indecomposable modules}
In this section, we will use the labels in section \ref{table} and explicitly list all Lie bracket relations among these generators.

Define $u_i=1_{K_i}$ to be the constructible function where $K_i$ is the orbit of indecomposable $U_i$ for $i=1,2,...,r$. 
Here $r$ is the number of nonisomorphic indecomposable modules.\\
The $(i,j)$-entry of the matrix is the Lie bracket $[u_i,u_j]=u_i\cdot u_j-u_j\cdot u_i$.
We only show the lower triangular part since the matrix satisfies $M^T=-M$.
We only need to compute brackets when the sum of the two dimension vectors is again the dimension vector of an indecomposable module, since otherwise the bracket must be zero.
\\
For $Q=A_1$:\\
There is only one indecomposable module so $\mathcal{A}\cong \mathbb{C}[x]$, where $x=1_{K}$ and $K$ is the orbit of the unique indecomposable module and the Lie algebra is just one dimensional.\\
For $Q=A_2$:\\
\begin{center}
  \begin{tabular}{ c|cccc}
$[u_i,u_j]$ &$u_1$  & $u_2$ & $u_3$ & $u_4$\\ \hline
 $u_1$&  &  &  & \\ 
 $u_2$&$u_4-u_3$  &  &  & \\ 
 $u_3$&0  &0  &  & \\ 
 $u_4$&0  &0  &0  & 
  \end{tabular}
\end{center}
For $Q=A_3$:\\
\begin{center}
  \begin{tabular}{ c|cccccccccccc}
$[u_i,u_j]$ &$u_1$  & $u_2$ & $u_3$ & $u_4$ &$u_5$  & $u_6$ & $u_7$ & $u_8$  &$u_9$  & $u_{10}$ & $u_{11}$ & $u_{12}$ \\ \hline 
$u_1$ &  &  &  &  &  &   &  &  &  &  &  & \\ 
$u_2$ &$u_6-u_4$  &  &  & &  &  &  &  &  &  &  & \\ 
$u_3$ &0  &$u_7-u_5$  &  &  &  &  &  &  &  &  & & \\ 
$u_4$ &0  &0  &$u_8-u_{10}$  &  &  &  &  &  &  &  &  & \\ 
$u_5$ &$u_{11}-u_8$  &0  &0  &$u_{12}$  &  &  &  &  &  &  &  & \\ 
$u_6$ &0  &0  &$u_{11}-u_9$  &0  &0  &  &  &  &  &  &  & \\ 
$u_7$ &$u_9-u_{10}$  &0  &0  &0  &0  &$-u_{12}$  &  &  & &  &  & \\ 
$u_8$ &0  &0  &0  &0  &0  &0  &0  &  &  &  &  & \\ 
$u_9$ &0  &0  &0  &0  &0  &0  &0  &0  &  &  &  & \\ 
$u_{10}$ &0  &$-u_{12}$  &0  &0  &0  &0  &0  &0  &0  & &  & \\ 
$u_{11}$ &0  &$u_{12}$  &0  &0  &0  &0  &0  &0  &0  &0  &  & \\ 
$u_{12}$ &0  &0  &0  &0  &0  &0  &0  &0  &0  &0  &0  &
  \end{tabular}
\end{center}

\newpage
For $Q=A_4$:\\
For the dimension vector reason, $[u_i,u_j]=0$ if $ 21\le i,j\le 40$.  
\begin{center}
  \begin{tabular}{ c|ccccccccccccccc}
 $[u_i,u_j]$&$u_1$  & $u_2$ & $u_3$ & $u_4$ &$u_5$  & $u_6$ & $u_7$ & $u_8$  &$u_9$  & $u_{10}$   \\  \hline 
$u_1$ &  &  &  &  &  &  &  &  &  &   \\ 
$u_2$ &$u_6-u_5$  &  & &  &  &  &  &  &  &    \\ 
$u_3$ &0  &$u_8-u_7$  &  &  &  &  &  &  &  &     \\ 
$u_4$ & 0 &0  & $u_{10}-u_9$  &  &  & &  &  &  &     \\ 
$u_5$ &0  &0  & $u_{11}-u_{13}$ &0  &  &  &  &  &  &    \\ 
$u_6$ &0 &0  & $u_{14}-u_{12}$ &0  &0  &  &  &  &  &    \\ 
$u_7$ & $u_{14}-u_{11}$ &0  &0  &$u_{15}-u_{17}$  &$u_{19}$  &0  &  &  &  &     \\ 
$u_8$ &$u_{12}-u_{13}$  &0  &0  &$u_{18}-u_{16}$  &0  &$-u_{19}$  &0  &  &  &     \\ 
$u_9$ & 0 & $u_{18}-u_{15} $ &0  &0  &$u_{23}-u_{21}$  &$u_{27}-u_{25}$  &$u_{20}$  &0  &  &    \\ 
$u_{10}$ & 0 & $u_{16}-u_{17}$  &0  &0  &$u_{24}-u_{22}$  &$u_{28}-u_{26}$  &0  &$-u_{20}$  &0  &     \\ \hline 
$u_{11}$ &0  & 0 &0  &$u_{21}-u_{22}$  &0  &0  &0  &0  &$-u_{32}$  &0     \\ 
$u_{12}$ & 0 & 0 &0  &$u_{27}-u_{28}$  &0  &0  &0  &0  &0  &$u_{31}$     \\ 
 $u_{13}$& 0 & $-u_{19}$ &0  &$u_{23}-u_{24}$  &0  &0  &0  &0  &0  &$u_{32}$    \\ 
$u_{14}$ & 0 & $u_{19}$ &0  &$u_{25}-u_{26}$  &0  &0  &0  &0  &$-u_{31}$  &0     \\ 
$u_{15}$ & $u_{25}-u_{21}$ & 0 &0  &0  &$u_{29}$ &0  &0  &0  &0  & 0   \\ 
$u_{16}$ & $u_{28}-u_{24}$ & 0 &0  &0  &0  &$-u_{30}$  &0  &0  &0  & 0   \\ 
$u_{17}$ &$u_{26}-u_{22}$  & 0 & $-u_{20}$ &0  &$u_{30}$  &0  &0  &0  &0  & 0   \\ 
$u_{18}$ &$u_{27}-u_{23}$  & 0 &$u_{20}$  &0  &0  &$-u_{29}$  &0  &0  &0  &0    \\ 
$u_{19}$ &0  & 0 &0  &$u_{29}-u_{30}$  &0  &0  &0  &0  &$u_{33}-u_{38}$  &$u_{37}-u_{34}$     \\ 
$u_{20}$ & $u_{31}-u_{32}$ & 0 &0  &0  &$u_{38}-u_{36}$  &$u_{35}-u_{37}$  &0  &0  &0  &0     \\ \hline 
$u_{21}$ &0  &0  &0  &0  &0  &0  &0  &0  &0  &0     \\
$u_{22}$ &  0& 0 &$-u_{32}$  &0  &0  &0  &$-u_{37}$ &$-u_{36}$  &0  &0     \\
$u_{23}$ & 0 & $-u_{29}$ &$u_{32}$  &0  &0  &0  &$u_{33}-u_{36}$  &0  &0  &0     \\
$u_{24}$ & 0 & $-u_{30}$ &0  &0  &0  &0  &$-u_{34}$  &$-u_{38}$  &0  &0     \\
$u_{25}$ & 0 & $u_{29}$  &0  &0  &0  &0  &$u_{37}$  &$u_{33}$  &0  &0     \\
$u_{26}$ &0  & $u_{30}$ &$-u_{31}$  &0  &0  &0  &0  &$u_{35}-u_{34}$  &0  &0     \\
$u_{27}$ & 0 & 0 &$u_{31}$  &0  &0  &0  &$u_{35}$  &$u_{38}$  &0  &0    \\
$u_{28}$ & 0 & 0 &0  &0  &0  &0  &0  &0  &0  &0     \\
$u_{29}$ &0  & 0 &$u_{37}-u_{33}$  &0  &0  &0  &0  &0  &0  &$-u_{40}$     \\
$u_{30}$ & 0 & 0 & $u_{34}-u_{38}$ &0  &0  &0  &0  &0  &$u_{40}$  &0     \\ \hline 
$u_{31}$ & 0 & $u_{38}-u_{35}$ &0  &0  &$-u_{39}$  &0  &0  &0  &0  &0     \\
$u_{32}$ & 0 & $u_{36}-u_{37}$ &0  &0  &0  &$u_{39}$  &0  &0  &0  &0    \\
$u_{33}$ & 0 &0  &0  &$-u_{40}$  &0  &0  &0  &0  &0  &0     \\
$u_{34}$ & 0 & 0 &0  &$u_{40}$  &0  &0  &0  &0  &0  & 0    \\ 
$u_{35}$ & $-u_{39}$  & 0 &0  &0  &0  &0  &0  &0  &0  &0   \\ 
$u_{36}$ & $u_{39}$  & 0 &0  &0  &0  &0  &0  &0  &0  &  0   \\ 
$u_{37}$ &0  & 0 &0  &0  &0  &0  &0  &0  &0  & 0    \\ 
$u_{38}$ &0  & 0 &0  &0  &0  &0  &0  &0  &0  & 0    \\ 
$u_{39}$ &0  & 0 &0  &0  &0  &0  &0  &0  &0  &  0   \\ 
$u_{40}$ &0  & 0 &0  &0  &0  &0  &0  &0  &0  & 0   \\
  \end{tabular}
\end{center}

\begin{center}
  \begin{tabular}{ c|ccccccccccccccc}
 $[u_i,u_j]$&$u_{11}$  & $u_{12}$ & $u_{13}$ & $u_{14}$ &$u_{15}$  & $u_{16}$ & $u_{17}$ & $u_{18}$  &$u_{19}$  & $u_{20}$    \\   \hline 
$u_{11}$ &  &  &  &  &  &  &  &  &  & \\ 
$u_{12}$ &0  &  & &  &  &  &  &  &  & \\ 
 $u_{13}$&0  &0  &  &  &  &  &  & &  & \\ 
$u_{14}$ &0  &0  &0  &  &  &  &  & & & \\ 
$u_{15}$ &$u_{37}$  &0  &$u_{33}$  &0  &  &  &  &  &  & \\ 
$u_{16}$ &0  &$-u_{38}$  &0  &$-u_{34}$  &0  &  &  &  &  & \\ 
$u_{17}$ &0  &$-u_{35}$  &$-u_{34}+u_{36}$  &$-u_{37}$  &0  &0  &  &  &  & \\ 
$u_{18}$ &$u_{36}$ &0  &$u_{38}$  &$-u_{35}+u_{33}$ &0  &0  &0  &  &  & \\ 
$u_{19}$ &0  &0  &0  &0  &0  &0  &0  &0  &  & \\ 
$u_{20}$ &0  &0  &0  &0  &0  &0  &0  &0  &0  & \\ \hline 
$u_{21}$ &0  &0  &0  &0  &0  &0  &0  &0  &0  &0 \\ 
$u_{22}$ &0  &$-u_{39}$  &0  &0  &0  &0  &0  &0  &0  &0 \\ 
 $u_{23}$&0  &0  &0  &$-u_{39}$  &0  &0  &$u_{40}$  &0  &0  &0 \\ 
$u_{24}$ &0  &0  &0  &0  &$-u_{40}$  &0  &0  &0  &0  &0 \\ 
$u_{25}$ &0  &0  &0  &0  &0  &$u_{40}$  &0  &0  &0  &0 \\ 
$u_{26}$ &0  &0  &$u_{39}$  &0  &0  &0  &0  &$-u_{40}$  &0  &0 \\ 
$u_{27}$ &$u_{39}$  &0  &0  &0  &0  &0  &0  &0  &0  &0 \\ 
$u_{28}$ &0  &0  &0  &0  &0  &0  &0  &0  &0  &0 \\ 
$u_{29}$ &0  &0  &0  &0  &0  &0  &0  &0  &0  &0 \\ 
$u_{30}$ &0  &0  &0  &0  &0  &0  &0  &0  &0  &0 \\ 
$u_{31}$ &0  &0  &0  &0  &0  &0  &0  &0  &0  &0 \\ 
$u_{32}$ &0  &0  &0  &0  &0  &0  &0  &0  &0  &0 \\ 
 $u_{33}$&0  &0  &0  &0  &0  &0  &0  &0  &0  &0 \\ 
$u_{34}$ &0  &0  &0  &0  &0  &0  &0  &0  &0  &0 \\ 
$u_{35}$ &0  &0  &0  &0  &0  &0  &0  &0  &0  &0 \\ 
$u_{36}$ &0  &0  &0  &0  &0  &0  &0  &0  &0  &0 \\ 
$u_{37}$ &0  &0  &0  &0  &0  &0  &0  &0  &0  &0 \\ 
$u_{38}$ &0  &0  &0  &0  &0  &0  &0  &0  &0  &0 \\ 
$u_{39}$ &0  &0  &0  &0  &0  &0  &0  &0  &0  &0 \\ 
$u_{40}$ &0  &0  &0  &0  &0  &0  &0  &0  &0  &0 \\ 
  \end{tabular}
\end{center}

\subsection{Lower and upper central series of the Lie algebra $\widetilde{\mathfrak{n}}$}
In this section we will show that $\widetilde{\mathfrak{n}}$ is nilpotent by computing its lower central series. Since we already know all the Lie brackets of $\widetilde{\mathfrak{n}}$ in terms of $u_i$, we can iteratively compute the series. Recall the definition of the \textbf{ lower central series} of a Lie algebra $\mathfrak{n}$: $\mathfrak{n}_{i+1}=[\mathfrak{n},\mathfrak{n}_{i}]$ and $\mathfrak{n}_0=\mathfrak{n}$.
Define generalized center of an ideal $\mathfrak{I}\subset \mathfrak{n}$ as $GC(\mathfrak{I}):=\{x\in \mathfrak{n}: [x,y]\in \mathfrak{I}, \forall y\in \mathfrak{n}  \}$. 
Then we can define the \textbf{upper central series} as follows: $\mathfrak{n}^{0}=0$, $\mathfrak{n}^{i+1}=GC(\mathfrak{n}^{i})$.
A Lie algebra $\mathfrak{n}$ is \textbf{nilpotent} if there is an integer $N$ such that $\mathfrak{n}_N=0$. 
The positive part $(\mathfrak{sl}_{n+1})_+$ of $\mathfrak{sl}_{n+1}$ is naturally a sub Lie algebra of $\widetilde{\mathfrak{n}}$ and the lower central series of $\frak{sl}_{n+1}$ in terms of basis elements $u_i$ is also computed. 
\\
For $Q=A_2$:\\
Lower central series:\\
\begin{align*}
\widetilde{\mathfrak{n}}_{1}&=\text{span}\{u_{3}-u_{4}\}, \\
\widetilde{\mathfrak{n}}_{2}&=0.
\end{align*}
\begin{center}
  \begin{tabular}{ |c|ccc|}
  \hline
  & $\widetilde{\mathfrak{n}}_0$ & $\widetilde{\mathfrak{n}}_1$ &$\widetilde{\mathfrak{n}}_2$  \\   \hline 
dim & 4 &1 & 0 \\ \hline
  \end{tabular}
\end{center}
Upper central series:\\
\begin{align*}
\widetilde{\mathfrak{n}}^{1}&=\text{span}\{u_{3},u_{4}\}, \\
\widetilde{\mathfrak{n}}^{2}&=\widetilde{\mathfrak{n}}.
\end{align*}
\begin{center}
  \begin{tabular}{ |c|ccc|}
  \hline
  & $\widetilde{\mathfrak{n}}^{0}$ & $\widetilde{\mathfrak{n}}^{1}$ &$\widetilde{\mathfrak{n}}^{2}$  \\   \hline 
dim & 0 &2 & 4 \\ \hline
  \end{tabular}
\end{center}
$(\mathfrak{sl}_3)_+$ as a sub Lie algebra:
\begin{align*}
((\mathfrak{sl}_3)_+)_{0}&=\text{span}\{u_{1},u_{2}\},\\
((\mathfrak{sl}_3)_+)_{1}&=\text{span}\{u_{3}-u_{4}\},\\
((\mathfrak{sl}_3)_+)_{2}&=0.
\end{align*}
For $Q=A_3$:\\
Lower central series:\\
\begin{align*}
\widetilde{\mathfrak{n}}_{1}&=\text{span}\{-u_{4}+u_{6},-u_{8}+u_{11},u_{9}-u_{10},-u_{5}+u_{7},-u_{12},u_{8}-u_{10}\}, \\
\widetilde{\mathfrak{n}}_{2}&=\text{span}\{u_{8}+u_{9}-u_{10}-u_{11},u_{12}\}, \\
\widetilde{\mathfrak{n}}_{3}&=0.
\end{align*}

\begin{center}
  \begin{tabular}{ |c|cccc|}
  \hline
  & $\widetilde{\mathfrak{n}}_0$ & $\widetilde{\mathfrak{n}}_1$ &$\widetilde{\mathfrak{n}}_2$  &$\widetilde{\mathfrak{n}}_3$\\   \hline 
dim & 12 &6 & 2  &0 \\ \hline
  \end{tabular}
\end{center}
Upper central series:\\
\begin{align*}
\widetilde{\mathfrak{n}}^{1}&=\text{span}\{u_{8},u_{9},u_{10}+u_{11},u_{12}\}, \\
\widetilde{\mathfrak{n}}^{2}&=\text{span}\{-u_4+u_6, -u_5+u_7, u_8,u_9,u_{10},u_{11},u_{12}\}, \\
\widetilde{\mathfrak{n}}^{3}&=\widetilde{\mathfrak{n}}.
\end{align*}

\begin{center}
  \begin{tabular}{ |c|cccc|}
  \hline
  & $\widetilde{\mathfrak{n}}^{0}$ & $\widetilde{\mathfrak{n}}^{1}$ &$\widetilde{\mathfrak{n}}^{2}$  &$\widetilde{\mathfrak{n}}^{3}$\\   \hline 
dim & 0 &4 & 7  & 12 \\ \hline
  \end{tabular}
\end{center}
$(\mathfrak{sl}_4)_+$ as a sub Lie algebra:
\begin{align*}
((\mathfrak{sl}_4)_+)_{0}&=\text{span}\{u_{1},u_{2},u_{3}\}\\
((\mathfrak{sl}_4)_+)_{1}&=\text{span}\{-u_{4}+u_{6},-u_{5}+u_{7}\}\\
((\mathfrak{sl}_4)_+)_{2}&=\text{span}\{u_{8}+u_{9}-u_{10}-u_{11}\}\\
((\mathfrak{sl}_4)_+)_{3}&=0
\end{align*}
For $Q=A_4$:\\
Lower central series:\\
\begin{align*}
\widetilde{\mathfrak{n}}_{1}&=\text{span}\{u_{5}-u_{6},u_{11}-u_{14},u_{13}-u_{12},u_{21}-u_{25},u_{24}-u_{28},u_{22}-u_{26},u_{23}-u_{27},u_{32}-u_{31},u_{39},u_{7}-u_{8},u_{15}-u_{18},\\
&u_{17}-u_{16},u_{19},u_{29},u_{30},u_{35}-u_{38},u_{36}-u_{37},u_{9}-u_{10},u_{11}-u_{13},u_{20},u_{32},u_{33}-u_{37},u_{34}-u_{38},u_{15}-u_{17},\\
&u_{21}-u_{22},u_{27}-u_{28},u_{40},u_{21}-u_{23},u_{36}-u_{38},u_{37}\}\\
\widetilde{\mathfrak{n}}_{2}&=\text{span}\{u_{11}+u_{12}-u_{13}-u_{14},u_{19},u_{21}-u_{23}-u_{25}+u_{27},u_{22}-u_{24}-u_{26}+u_{28},u_{29},u_{30},\\
&u_{35}+u_{36}-u_{37}-u_{38},u_{39},u_{21}-u_{22}-u_{25}+u_{26},u_{31}-u_{32},u_{37},u_{34},u_{33}-u_{35}-u_{36},u_{33},\\&u_{40},u_{15}+u_{16}-u_{17}-u_{18},u_{20},u_{36}-u_{37},-u_{21}+u_{22}+u_{23}-u_{24},u_{32}\}\\
\widetilde{\mathfrak{n}}_{3}&=\text{span}\{u_{21}-u_{22}-u_{23}+u_{23}-u_{25}+u_{26}-u_{27}+u_{28},u_{31}-u_{32},\\
&u_{33}-u_{37},u_{34}-u_{38},-u_{34}+u_{35}+u_{36}-u_{37},u_{39},u_{29}-u_{30},u_{33}-u_{38},u_{40},u_{36}-u_{37}\}\\
\widetilde{\mathfrak{n}}_{4}&=\text{span}\{-u_{33}-u_{34}+u_{35}+u_{36},u_{39},u_{40},-u_{35}-u_{36}+u_{37}+u_{38}\}\\
\widetilde{\mathfrak{n}}_{5}&=0\\
\end{align*}
\begin{center}
  \begin{tabular}{ |c|cccccc|}
  \hline
  & $\widetilde{\mathfrak{n}}_0$ & $\widetilde{\mathfrak{n}}_1$ &$\widetilde{\mathfrak{n}}_2$  &$\widetilde{\mathfrak{n}}_3$ &$\widetilde{\mathfrak{n}}_4$ &$\widetilde{\mathfrak{n}}_5$ \\   \hline 
dim & 40 &30 & 20  &10 & 4  &0 \\\hline
  \end{tabular}
\end{center}
Upper central series:\\
\begin{align*}
\widetilde{\mathfrak{n}}^{1}&=\text{span}\{u_{40},u_{39},u_{38},u_{37},u_{35}+u_{36},u_{33}+u_{34},u_{28},u_{21}\},\\
\widetilde{\mathfrak{n}}^{2}&=\text{span}\{u_{40},u_{39},u_{38},u_{37},u_{36},u_{35},u_{34},u_{33},u_{31}-u_{32},u_{29}-u_{30},u_{28},-u_{22}-u_{23}+u_{24}-u_{25}+u_{26}+u_{27},u_{21}\},\\
\widetilde{\mathfrak{n}}^{3}&=\text{span}\{u_{40},u_{20},u_{19},-u_{15}-u_{16}+u_{17}+u_{18},-u_{11}-u_{12}+u_{13}+u_{14},u_{39},u_{38},u_{37},u_{36},u_{35},u_{34},\\
&u_{33},u_{32},u_{31},u_{30},u_{29},u_{28},u_{22}+u_{27},u_{23}+u_{26},u_{22}+u_{23}+u_{25},-u_{22}-u_{23}+u_{24},u_{21}\},\\
\widetilde{\mathfrak{n}}^{4}&=\text{span}\{u_{19},u_{20},u_{21},u_{22},u_{23},u_{24},u_{25},u_{26},u_{27},u_{28},u_{29},u_{30},u_{31},u_{32},u_{33},u_{34},u_{35},u_{36},u_{37},u_{38},u_{39},u_{40}
,\\
&u_{15}-u_{17},u_{15}-u_{16},u_{11}-u_{14},u_{11}-u_{13},u_{11}-u_{12},u_{9}-u_{10},u_{7}-u_{8},u_{5}-u_{6},u_{15}-u_{18}\},\\
\widetilde{\mathfrak{n}}^{5}&=\widetilde{\mathfrak{n}}\\
\end{align*}
\begin{center}
  \begin{tabular}{ |c|cccccc|}
  \hline
  & $\widetilde{\mathfrak{n}}^0$ & $\widetilde{\mathfrak{n}}^1$ &$\widetilde{\mathfrak{n}}^2$  &$\widetilde{\mathfrak{n}}^3$ &$\widetilde{\mathfrak{n}}^4$ &$\widetilde{\mathfrak{n}}^5$ \\   \hline 
dim & 0 &8 & 13  &22 & 31   & 40 \\\hline
  \end{tabular}
\end{center}
$(\mathfrak{sl}_5)_+$ as a sub Lie algebra:
\begin{align*}
((\mathfrak{sl}_5)_+)_{0}&=\text{span}\{u_{1},u_{2},u_{3},u_{4}\}\\
((\mathfrak{sl}_5)_+)_{1}&=\text{span}\{-u_{5}+u_{6},-u_{7}+u_{8},-u_{9}+u_{10}\}\\
((\mathfrak{sl}_5)_+)_{2}&=\text{span}\{u_{11}+u_{12}-u_{13}-u_{14},u_{15}+u_{16}-u_{17}-u_{18}\}\\
((\mathfrak{sl}_5)_+)_{3}&=\text{span}\{-u_{21}+u_{22}+u_{23}-u_{24}+u_{25}-u_{26}-u_{27}+u_{28}\}\\
((\mathfrak{sl}_5)_+)_{4}&=0
\end{align*}

\subsection{Lie algebra cohomology of $\widetilde{\mathfrak{n}}$}
In this section, we will compute the Lie algebra cohomology of $\widetilde{\mathfrak{n}}$.  
The Lie algebra cohomology is graded by cohomological dimension and the the dimension vector of $\Lambda$-module. 
The dimension of each graded piece will be given.
Since $\widetilde{\mathfrak{n}}$ is nilpotent, the Lie algebra cohomology will satisfy Poincare duality.
The Lie algebra cohomology for the case $Q=A_4$ is too big to calculate and is not given here.
\\
For $Q=A_2$:\\
\begin{center}
  \begin{tabular}{ |c|ccccc|}
  \hline
  & $H^0$ & $H^1$ &$H^2$  &$H^3$ &$H^4$  \\   \hline 
dim & 1 &3 & 4  &3 & 1    \\\hline
  \end{tabular}
\end{center}
\begin{figure}[ht!]
\centering
\includegraphics[width=90mm]{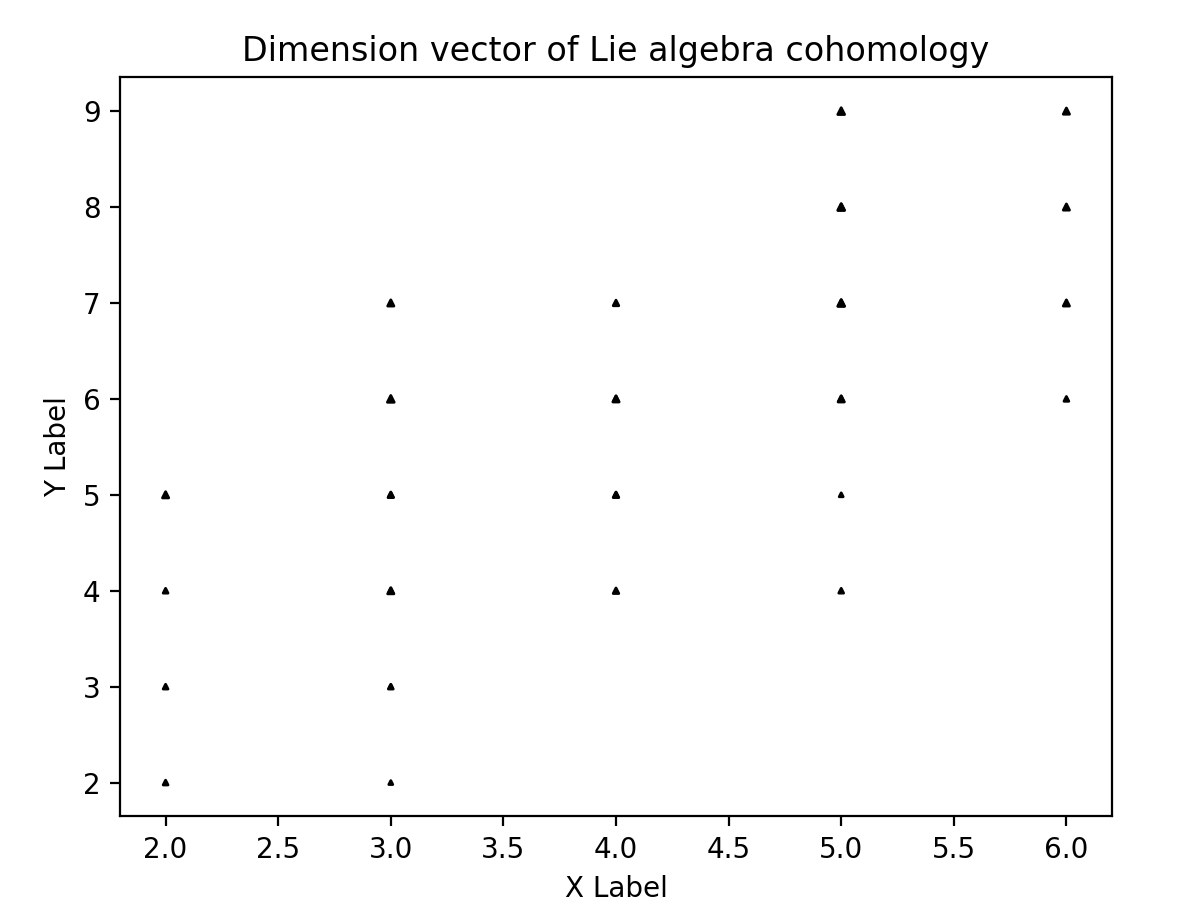}
\caption{  \label{overflow}}
\end{figure}
For $Q=A_3$:\\
\begin{center}
  \begin{tabular}{ |c|ccccccccccccccc|}
  \hline
  & $H^0$ & $H^1$ &$H^2$  &$H^3$ &$H^4$& $H^5$ & $H^6$ &$H^7$  &$H^8$ &$H^9$& $H^{10}$ & $H^{11}$ &$H^{12}$  \\  \hline 
dim & 1&6&20&47&85&121&136&121&85&47&20&6&1      \\\hline
  \end{tabular}
\end{center}
\begin{figure}[ht!]
\centering
\includegraphics[width=120mm]{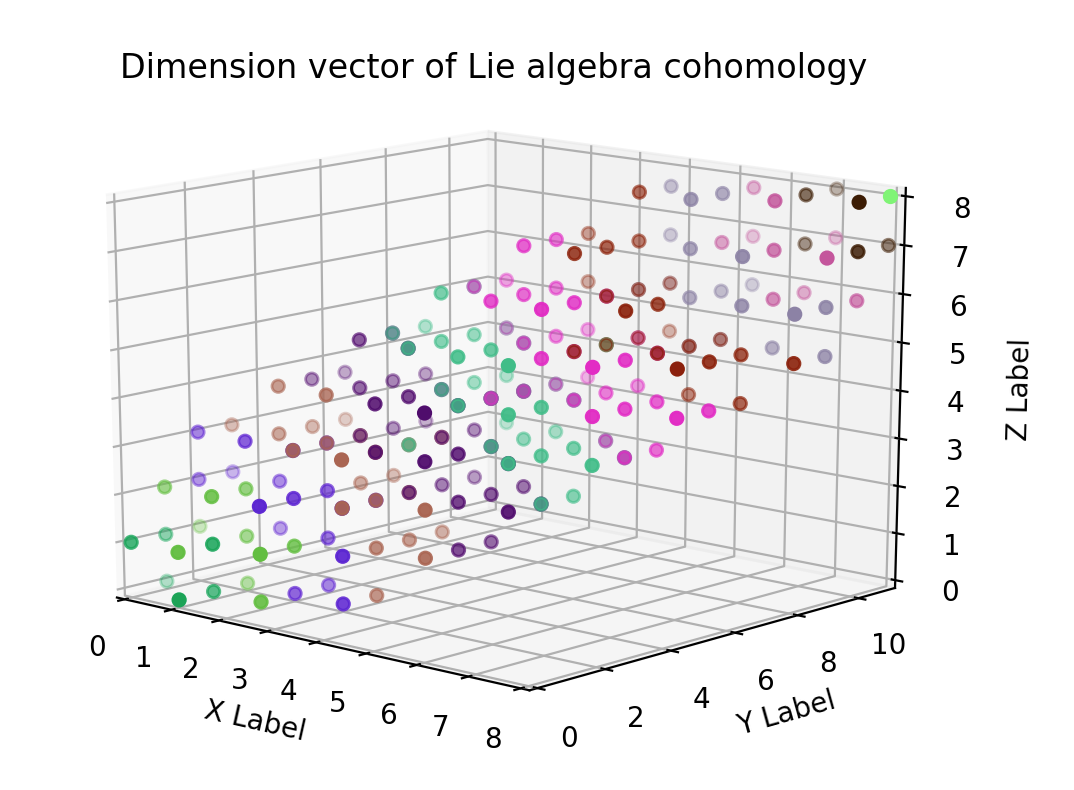}
\caption{  \label{overflow}}
\end{figure}

\end{document}